\documentclass[12pt,fleqn]{amsart}
\usepackage[utf8]{inputenc}
\usepackage[english]{babel}
\usepackage{color}
\usepackage{indentfirst}
\usepackage{amssymb}
\usepackage{amsthm}

\usepackage{color}
\usepackage{mathtools,amsmath,amssymb,amsthm,mathrsfs}
\usepackage{enumerate}
\newtheorem{thm}{Theorem}
\newtheorem{lemma}[thm]{Lemma}

\newtheorem{mydef}[thm]{Definition}

\newtheorem{cor}[thm]{Corollary}

\numberwithin{equation}{thm}
\numberwithin{thm}{section}
\newtheorem{question}[thm]{Question}

\newcommand{\cA}{{\mathcal A}}
\newcommand{\cB}{\protect{\mathcal B}}

\newcommand{\cN}{\mathcal N}
\newcommand{\cM}{\mathcal M}
\newcommand{\cP}{{\mathcal P}}

\newcommand{\cU}{{\mathcal U}}

\newcommand{\fA}{\mathfrak A}
\newcommand{\fB}{\mathfrak B}
\newcommand{\fC}{\mathfrak C}
\newcommand{\fD}{\mathfrak D}

\newcommand{\fM}{\mathfrak M}
\newcommand{\fb}{\mathfrak b}
\newcommand{\fc}{\mathfrak c}

\newcommand{\btu}{\bigtriangleup}

\newcommand{\wh}{\widehat }
\newcommand{\ult}{{\rm ult}}
\newcommand{\vf}{\varphi}

\newcommand{\stevo}{Todor\v{c}evi\'c}

\newcommand{\sm}{\setminus}

\newcommand{\sub}{\subseteq}

\newcommand{\clop}{\protect{\rm Clopen} }

\newcommand{\fin}{\mbox{\it fin}}
\newcommand{\eps}{\varepsilon}

\DeclareMathOperator{\alg}{alg}
\DeclareMathOperator{\ba}{ba}

\newcommand{\la}{\langle}
\newcommand{\ra}{\rangle}
\newcommand{\growth}{\gamma\omega\!\sm\!\omega}
\newcommand{\er}{\mathbb R}
\title[Compactifications of $\omega$ and $c_0$]{Compactifications of $\omega$ and the Banach space $c_0$}
\author[P. Drygier]{Piotr Drygier}
\author[G. Plebanek]{Grzegorz Plebanek}

\address{Instytut Matematyczny\\ Uniwersytet Wroc\l awski\\ Pl.\ Grunwaldzki 2/4\\
50-384 Wroc\-\l aw\\ Poland} \email{piotr.drygier@math.uni.wroc.pl, grzes@math.uni.wroc.pl}

\thanks{The second author was partially supported by NCN grant 2013/11/B/ST1/03596 (2014-2017).}

\begin{document}

\begin{abstract}
We investigate for which compactifications $\gamma\omega$ of the discrete space of natural numbers $\omega$,
 the natural copy of the Banach space $c_0$ is complemented in $C(\gamma\omega)$.
We show, in particular, that the separability of the remainder $\gamma\omega\sm\omega$ is neither sufficient nor
necessary for $c_0$ being complemented in $C(\gamma\omega)$ (for the latter our result is proved
under the continuum hypothesis).
We analyse, in this context,  compactifications of $\omega$ related to embeddings  of the measure algebra into $P(\omega)/\fin$.

We also prove that a Banach space $C(K)$ contains a rich family of complemented copies of $c_0$ whenever the compact space $K$ admits
only measures of countable Maharam type.
\end{abstract}

\maketitle%\thispagestyle{fancy}

\section{Introduction}

If $X$ is a Banach space and $Y$ is a closed subspace of $X$ then $Y$ is said to be complemented in $X$ if there is a closed subspace $Z$ of $X$
such that $X=Y \oplus Z$. This is equivalent to saying that there is a bounded linear operator $P$ from $X$ onto $Y$ which is a projection, i.e.\
$P\circ P=P$. Recall that typically,  unless a Banach space $X$ is isomorphic to a Hilbert space, there are many closed uncomplemented subspaces of $X$.

The classical Banach space $c_0$ plays a special role when we speak of complementability:
by Sobczyk's theorem \cite{So41} every isomorphic copy of $c_0$ is complemented in any separable superspace.
 Cabello Sanchez,  Castillo and Yost \cite{SCY00} offer an interesting discussion of various proofs and aspects of Sobczyk's theorem; see also
 a survey paper by Godefroy \cite{Go01}.

Complementability of isomorphic copies of $c_0$ has been investigated for nonseparable spaces.
 A Banach space $X$ is said to have the {\em Sobczyk property} if every subspace of $X$ isomorphic to $c_0$ is complemented in $X$.
Molt\'o \cite{Mo91} singled out a certain topological property of the $weak^*$ topology in $X^*$ ensuring that $X$ has the Sobczyk property.
Correa and Tausk \cite{CT14} proved that the space $C(K)$ has the Sobczyk property whenever $K$ is a compact line (generalizing
an earlier  result from \cite{Pa93}, where the same was proved for $K$ being the double arrow space);
see also \cite{ACGJM02}, \cite{CK15}, \cite{FKLW11}, \cite{GP03} for related results.

Let $\gamma\omega$ be a compactification of the discrete space $\omega$ of natural numbers. Then $c_0$ can be naturally identified with  the subspace $Y$ of $C(\gamma\omega)$,
where
\[Y=\{f\in C(\gamma\omega): f|(\gamma\omega\sm\omega)\equiv 0\},\]
simply by identifying the unit vector $e_n$ in $c_0$ with $\chi_{\{n\}}\in C(\gamma\omega)$.
In the sequel, we shall call the space $Y$ the {\em natural copy} of $c_0$ in $C(\gamma\omega)$.
We also use the following terminology.

\begin{mydef}
We say that a compactification $\gamma\omega$ is {\em smooth} if the natural copy of $c_0$ is complemented in $C(\gamma\omega)$.
\end{mydef}

The main problem that is considered in the present paper may be stated informally as follows.

\begin{question}\label{question}
Which compactifications $\gamma\omega$ are smooth?
\end{question}

Note that every metrizable compactification  $\gamma\omega$ is smooth because $C(\gamma\omega)$ is then separable and {\em every} copy of $c_0$
is complemented in $C(\gamma\omega)$ by Sobczyk's theorem.  On the other hand,  the maximal compactification $\beta\omega$ is not smooth:
$C(\beta\omega)$ is isometric to $l_\infty$ and, by Phillips' theorem \cite{Ph40}, $c_0$ is not complemented in $l_\infty$.
In fact $C(\beta\omega)$ is a Grothendieck space so it  contains no complemented copies of $c_0$ (see the next section).
Note also that if we have two comparable compactification $\gamma_1\omega\leq \gamma_2\omega$, in the sense that there is a continuous surjection $\gamma_2\omega\to \gamma_1\omega$ that does not move points from $\omega$, then  $\gamma_1\omega$ is smooth provided $\gamma_2\omega$ is smooth.
Thus smooth compactifications form a natural subclass of all compactifications of $\omega$ and \ref{question} calls for a reasonable characterization of smoothness.

Question \ref{question} has been  motivated by Castillo \cite{Ca01} and by conversations with Wies{\l}aw Kubi\'s and Piotr Koszmider.
In particular, W.\ Kubi\'s observed that if $\gamma\omega$ is smooth then the remainder $\gamma\omega\sm\omega$ must carry a strictly positive measure (see section
\ref{sec:5}), and asked if the converse implication holds.

The paper is organized as follows.
In section \ref{sec:2} we recall the standard facts related to complementability of $c_0$. In section \ref{sec:3} we introduce the terminology and notation concerning Boolean algebras
and finitely additive measures and then translate facts form section \ref{sec:2} to the setting that is used throughout the paper.

In section \ref{sec:4} we consider compactifications of $\omega$ defined by subalgebras $\fA$ of $P(\omega)$ containing all finite sets and such that
the quotient map $\fA\to\fA/\fin$ admits a lifting. We prove in particular that every separable zerodimensional compact space is homeomorphic to the remainder of a smooth
compactification (Theorem \ref{4:3}).

Our main results read as follows.

\begin{enumerate}
\item If we take a compactification $\gamma\omega$ related to an embedding of the measure algebra into $P(\omega)/\fin$
then $\gamma\omega$ is not smooth (Theorem \ref{5:1}).
Since in such a case $\growth$ is homeomorphic to the Stone space of the measure algebra, $\growth$ does carry a strictly
positive measure. A~related result on such $\gamma\omega$ is a content of  Theorem \ref{5:5}.
\item Under  CH there is a smooth compactification with a nonseparable remainder, see  Theorem \ref{6:1}.
\item There is a non-smooth compactification of $\omega$ with a separable remainder, see  Theorem \ref{7:1}.
\end{enumerate}

The conclusion is that, as it seems,  the smoothness of $\gamma\omega$ is not directly related to simple topological properties of
$\gamma\omega\sm\omega$. In fact,   a smooth compactification may have the same remainder as another non-smooth one,
see Corollary \ref{4:4}.
We do not know if  (2) above is provable in the usual set theory;
it is likely that for our argument we can relax CH to the assumption $\fb=\fc$. However, we have not been able to
show without additional set-theoretic assumptions a formally weaker assertion:
{\em there is a compactification of $\omega$ with a nonseparable remainder that carries a strictly positive measure},
cf.\ Drygier and Plebanek \cite{DP15}.

In the final section we prove a general result on  $C(K)$ spaces containing complemented copies of $c_0$. Theorem \ref{8:4} says that
if a compact space $K$ has a certain measure theoretic property then every isomorphic copy of $c_0$ inside $C(K)$ contains a complemented subcopy
of $c_0$. Our result  is related to the work of  Molt\'o \cite{Mo91}  and Galego and Plichko  \cite{GP03}.

\section{Preliminaries}\label{sec:2}

In the sequel, $K$ (possibly with some subscript) always denotes a compact Hausdorff space and $C(K)$ stands for
the Banach space of (real-valued) continuous functions on $K$ equipped with the usual supremum norm.
The dual space $C(K)^*$ is identified with the space $\cM(K)$ of all signed Radon measures of bounded variation defined on the Borel $\sigma$-algebra
on $K$. For $\mu\in \cM(K)$ and $f\in C(K)$ we write $\mu(g)=\int_K f\;{\rm d}\mu$ for simplicity.
Recall that every $\mu\in\cM(K)$ can be written as $\mu=\mu^+-\mu^-$, where $\mu^+,\mu^-$ are nonnegative orthogonal measures.
Then $|\mu|$, the total variation of $\mu$, is defined as $|\mu|=\mu^+ + \mu^-$,
and  the norm of $\mu$ is  $||\mu||=|\mu|(K)$.  If $x\in K$ then $\delta_x\in \cM(K)$ denotes the probability Dirac measure at $x$.
The basic facts on $C(K)$ and $\cM(K)$ may be found in Albiac and Kalton \cite{AK05} or Diestel \cite{Di84}.

 The following well-known lemma,  establishing a connection between sequences of  measures on $K$ and
complementability of $c_0$ in $C(K)$   originates in Veech's proof of Sobczyk's theorem \cite{Ve71}.
We write here $\delta(n,k)$ for the Kronecker symbol.

\begin{lemma}\label{2:1}
  Let $T\colon c_0 \to C(K)$ be an isomorphic embedding and let $Te_n = \varphi_n$.
  Then the following conditions are equivalent:
  \begin{enumerate}[(i)]
    \item $T[c_0]$ is complemented in $C(K)$;
    \item there exist   bounded sequences $(\mu_n)_n$ and  $(\nu_n)_n$ in $ \cM(K)$
    such that
    \begin{enumerate}[---]
      \item $\nu_n(\varphi_k) = 0$ for every  $n,k\in\omega$,
      \item $\mu_n(\varphi_k) = \delta(n,k)$  for every  $n,k\in\omega$,
      \item $\mu_n - \nu_n \to 0$ in the weak$^*$ topology.
    \end{enumerate}
  \end{enumerate}
\end{lemma}

\begin{proof}
 To check   $(ii) \Rightarrow (i)$ define $P\colon C(K) \to C(K)$ by
  \[ Pf = \sum_{n\in\omega} (\mu_n - \nu_n)(f)\cdot  \varphi_n. \]
Then $P$   is easily seen to be a bounded projection from $C(K)$ onto $T[c_0]$.

  For the converse implication consider the dual operator
  $T^*\colon\cM(K)\to c_0^* = l_1$.
  Since $T$ is an isomorphic embedding,
  $T^*$ is a surjection so for each $e_n^* = e_n\in l_1$ there exists a measure
  $\mu_n\in\cM(K)$ such that
  $T^* \mu_n = e_n^*$,  and the sequence of $\mu_n$ is norm bounded.
  We have  $T^*(\mu_n)(e_k) = e_n^* (e_k) = \delta(n,k)$ and
   $T^*(\mu_n)(e_k) = \mu_n (T e_k) = \mu_n(\varphi_k)$, so
  $\mu_n(\varphi_k) = \delta(n,k)$.
  For every $n$ define a measure $\nu_n\in\cM(K)$ putting
  $\nu_n(f) = \mu_n(f) -\mu_n(Pf)$ for $f\in C(K)$. Then $\nu_n$ vanishes on $P[C(K)]$ and
  for every $f\in C(K)$, taking $x\in c_0$ such that $Tx=Pf$, we get
  \[\mu_n(f)-\nu_n(f)=\mu_n(Pf)=\mu_n(Tx)=e_n^*(x)\to 0,\]
as required.
\end{proof}

Here is the most obvious illustration of the previous lemma.

\begin{cor}\label{2:2}
  If $K$ contains a non-trivial converging sequence then $C(K)$ contains a complemented copy of $c_0$.
\end{cor}

\begin{proof}
  Let $(x_n)_n$ be a sequence in  $K$ converging to $x\in K$ and  such that $x_n \neq x$ for every $n$.
 Then it is easy to construct a  pairwise disjoint family $\{U_n\colon n\in\omega\}$ of open subsets
    of $K$ such that $x_n \in U_n$ for each $n\in\omega$.

 For every $n$ we can find a continuous function $f_n:K\to [0,1]$ such that $f_n(x_n)=1$ and $f_n$ vanishes outside
 $U_n$. Now if we define $T:c_0\to C(K)$ by $Te_n=f_n$ then $T[c_0]$ is complemented in $C(K)$. Indeed, we
 can apply Lemma \ref{2:1} with $\mu_n=\delta_{x_n}$ and $\nu_n=\delta_x$ for every $n$.
 \end{proof}

The way we stated Lemma \ref{2:1} is motivated by its application to  compactifications of $\omega$.

\begin{cor}\label{2:3}
  A compactification $\gamma\omega$ is smooth if and only if there exists
    a bounded sequence of measures $(\nu_n)_n$ in $\cM(\gamma\omega)$
  such that $|\nu_n|(\omega)=0$ for every $n$ and  $\nu_n-\delta_n\to 0$ in the $weak^*$ topology.
\end{cor}

Let us note that the smoothness of $\gamma\omega$ is directly related to the existence of a certain extension operator.
If $L$ is a closed subspace of a compact space $K$ then an extension operator $E:C(L)\to C(K)$ is a bounded linear operator
such that $E(f)|L=f$ for every $f\in C(L)$; see Avil\'es and Marciszewski \cite{AM15} for a recent result on extension operators
and references therein.

\begin{lemma}\label{2:3.1}
A compactification $\gamma\omega$ is smooth if and only if there is an extension operator $C(\growth)\to C(\gamma\omega)$.
\end{lemma}

\begin{proof}
Suppose that $P:C(\gamma\omega)\to c_0$ is a bounded projection (where $c_0$ is identified with its natural copy inside $C(\gamma\omega)$.
For $f\in C(\growth)$ take any extension of $f$ to a function $g\in C(\gamma\omega)$ and define $E(f)=g-Pg$.
Note that $E(f)$ is uniquely defined: if $g'\in C(\gamma\omega)$ is another extension of $f$ then $g'-g$ vanishes on the remainder so
$g'-g\in c_0$ and $P(g'-g)=g'-g$, that is $g'-Pg'=g-Pg$.

Suppose now that $E:C(\growth)\to C(\gamma\omega)$ is an extension operator. Then $Pg=g-E(g|\growth)$ defines a projection from
$C(\gamma\omega)$ onto $c_0$.
\end{proof}

We shall now recall the notion of a Grothendieck space  which for $C(K)$ spaces means being anti-Sobczyk.

\begin{mydef}
  A Banach space is said to be a \emph{Grothendieck space} if every
    $weak^*$ null sequence $(x^*_n)_n$ in $ X^*$ converges weakly (i.e.\ $x^{**}(x^*_n)\to 0$ for every
    $x^{**}\in X^{**}$).
\end{mydef}

For the proof of the following see Cembranos \cite{Ce84}.

\begin{thm}\label{2:4}
 Given a  compact space $K$,  the space $C(K)$ is Grothendieck  if and only if
    $C(K)$  does not contain a complemented copy of $c_0$.
\end{thm}

Recall that typical examples of Grothendieck spaces are $l_\infty$ and, more generally, $C(K)$ spaces
where $K$ is extremally disconnected compact space, see \cite {AK05} or \cite{Di84}.

\section{Boolean algebras and compactifications}\label{sec:3}

We shall consider mainly zerodimensional compactifications of $\omega$ and those are naturally
related to Boolean subalgebras of $P(\omega)$. If $\fA$ is any Boolean algebra then $\ult(\fA)$ denotes its Stone space of all ultrafilters
on $\fA$. If $a\in\fA$ then $\wh{a}$ is the corresponding clopen set in $\ult(\fA)$, that is
\[\wh{a}=\{x\in \ult(\fA): a\in x\}.\]
A family $\cU\sub p$ is a base of the ultrafilter $p$ if every $A\in p$ contains some $U\in\cU$, in other words,
if $\{\wh{U}:U\in \cU\}$ is a local base at $p\in \ult(\fA)$.

Let $\fA$ be any Boolean algebra.
Then $\ba(\fA)$ will stand for the family of all finitely additive measures $\mu$ on $\fA$ that have bounded variation and $\ba_+(\fA)$ are finitely additive nonnegative
functions. We call any $\mu\in\ba(\fA)$ simply a measure. We now recall the following standard facts, see e.g.\ Semadeni \cite{Se71}, 18.7.
Every $\mu\in\ba_+(\fA)$ can be transferred onto the algebra of clopen subsets of $\ult(\fA)$ by the formula
$\wh{\mu}(\wh{a})=\mu(a)$, and then uniquely extended to a Radon measure on $\ult(\cA)$ (that Radon measure is  still denoted by $\wh{\mu}$).
Note  that the $weak^*$ topology on a bounded subset of $\cM(\ult(\fA))$ may be seen as the topology of convergence of
clopen subsets of $\ult(\fA)$. Hence for a bounded sequence $\mu_n$ in $\ba(\fA)$  we have $\wh{\mu_n}\to 0$ in the $weak^*$ topology of
$\cM(\ult(\fA))$ if and only if $\mu_n(a)\to 0$ for every $a\in\fA$.

If $\fA$ is a subalgebra of $P(\omega)$ and $\fA$ contains $\fin$, the family of all finite subsets of $\omega$, then $\ult(\fA)$ is a compactification of
$\omega$ --- we simply identify points in $\omega$ with principal ultrafilters  on $\fA$. We shall denote such a  compactification of $\omega$ by
$K_\fA$ and $K_\fA^*=K_\fA\sm\omega$ will stand for its remainder. Note that $K_\fA^*$ may be identified with $\ult(\fA/\fin)$.

Using the terminology and notation introduced above  we can rewrite  Corollary \ref{2:3} as follows.

\begin{lemma}\label{3:1}
Let $\fA$ be an algebra such that $\fin\sub\fA\sub P(\omega)$.
Then the compactification $K_\fA$ of $\omega$ is smooth if and only if there exists a bounded
sequence $(\nu_n)_n$ in $\ba(\fA)$  such that

\begin{enumerate}[(i)]
\item $\nu_n|\fin \equiv 0$ for every $n$, and
 \item $\nu_n-\delta_n\to 0$ on $\fA$, that is $(\nu_n - \delta_n)(A) \to 0$ for every $A\in\fA$.
 \end{enumerate}
\end{lemma}

Note that in case $A\in\fA$ is infinite and co-infinite, condition (ii) above is equivalent to
\vspace{-2ex}

\begin{equation}\label{eq:conv}
  \lim_{ n\in A} \nu_n (A) = 1,
  \qquad \lim_{ n\notin A} \nu_n(A) = 0.
\end{equation}

We shall below often enlarge a given algebra $\fA\sub P(\omega)$ by adding a new set $X\subseteq \omega$;
let $\fA[X]$ be the algebra generated by $\fA$ and $X$. Then
\[  \fA[X] = \big\{(A\cap X) \cup (A'\sm X)\colon A,A' \in \fA\big\}.\]

If $\mu\in \ba_+(\fA)$ then $\mu_*$ and $\mu^*$ denote the corresponding inner and outer measure defined as
\[\mu_*(X)=\sup\{\mu(A): A\in\fA, A\sub X\},\quad \mu^*(X)=\inf\{\mu(A): A\in\fA, A\supseteq X\}.\]

The following fact on extensions of finitely additive measures is due to {\L}o\'s and Marczewski \cite{LM48}.

\begin{thm}\label{3:2}
Given an algebra $\fA$, $\mu\in\ba_+(\fA)$ and any $X$ the following formulas define extensions of
$\mu$ to $\mu_1,\mu_2\in \ba_+(\fA[X]$
\[ \mu_1\left( (A\cap X)\cup(A'\setminus X)\right) =\mu_*(A\cap X) + \mu^*(A'\sm X),\]
\[ \mu_2\left( (A\cap X)\cup(A'\setminus X)\right) =\mu^*(A\cap X) + \mu_*(A'\sm X).\]
Consequently, for every $t$ satisfying $\mu_*(X)\le t\le \mu^*(X)$, there is an extension of $\mu$ to
$\mu_t\in \ba_+(\fA[X])$ such that $\mu_t(X)=t$.
\end{thm}

The following  lemma  on convergence of extended probability measures will be needed in section \ref{sec:6}.

\begin{lemma}\label{3:3}
Let $\fA$ be an algebra such that $\fin\sub\fA\sub P(\omega)$ and  let  $(\nu_n)_n$ be  a   sequence
of probability measures from $\ba_+(\fA)$ such that $\nu_n - \delta_n \to 0$.
Further let, for every $n$, $\widetilde\nu_n\in\ba_+(\fA[X])$ be any extension of $\nu_n$.

If $\widetilde \nu_n(X) - \delta_n (X) \to 0$
    then $\widetilde \nu_n - \delta_n \to 0$ on  $\fA[X]$.
\end{lemma}

\begin{proof}
We use here  \ref{eq:conv}. For any $A\in\fA$ if $n$ runs through $A\cap X$ then $\widetilde\nu_n(X)\to 1$ and
$\widetilde\nu_n(A)=\nu_n(A)\to 1$ so $\widetilde\nu_n(A\cap X)\to 1$ (using  $\widetilde \nu_n(\omega)=\nu_n(\omega)=1$).

Take any $\eps>0$. Then $\nu_n(A)<\eps$ if $n\notin A$ and $n\ge n_0$ and $\widetilde\nu_n(X)<\eps$ whenever
$n\ge n_1$ and $n\notin X$. Hence for $n\ge\max(n_0,n_1)$, if $n\notin A\cap X$ then either
$n\notin A$ and $\widetilde\nu_n(A\cap X)\le \nu_n(A)<\eps$ or $n\notin X$ and $\widetilde\nu_n(A\cap X)\le \widetilde\nu_n(X)<\eps$.

The convergence of $\widetilde\nu_n (A\sm X)$ may be checked in a similar way.
\end{proof}

Recall that a nonnegative measure $\mu\in\ba(\fA)$ is said to be {\em nonatomic} if for every $\eps>0$ there is a finite partition
of $1_\fA$ into pieces of measure $<\eps$. A signed measure $\mu$ is nonatomic if its variation $|\mu|$ is nonatomic.
We shall use the following two simple observations.

\begin{lemma}\label{3:4}
Given a Boolean algebra $\fA$ and a signed measure $\mu$ on $\fA$,
$\mu$ is nonatomic if and only if $\inf\{|\mu|(A): A\in p\}=0$ for every $p\in\ult(\fA)$, and this is equivalent to saying that
$\wh{\mu}$ is a Radon measure on $K_\fA$ vanishing on points of $K_\fA$.

If $\mu$ is nonatomic on $\fA$, $a\in\fA$ and $t<\mu(a)$ then for every $\eps>0$ there is $b\in\fA$ such that
$b\leq a$ and $|\mu(b)-t|<\eps$.
\end{lemma}

\section{Liftings}\label{sec:4}

Let $\fA$ be an algebra such that $\fin\sub\fA\sub P(\omega)$. Consider the canonical quotient map
\[\fA\to\fA/\fin$, $\fA\ni A\to A^\bullet\in\fA/\fin.\]
By a {\em lifting} of the quotient map we mean  a Boolean homomorphism $\rho\colon\fA/\fin\to\fA$ such that
$\rho(a)^\bullet=a$ for every $a\in\fA/\fin$.

\begin{lemma}\label{4:1}
For an algebra $\fA$ such that $\fin\sub\fA\sub P(\omega)$, the quotient map $\fA\to\fA/\fin$ admits a lifting if and only
there exists an Boolean algebra
  $\fB\sub P(\omega)$ such that every  $B\in\fB\sm\{\emptyset\}$ is infinite and
 $\fA$ is equal to $\alg(\fB\cup\fin)$, the algebra generated by $\fB$ and $\fin$.
\end{lemma}

\begin{proof}
  If $\rho$ is a lifting, then put $\fB = \rho[\fA / \fin]$. For every nonzero $a\in\fA/\fin$ the set $\rho(a)$ is infinite because
  $\rho(a)^\bullet=a\neq 0$.

  If $\fA = \alg(\fB \cup \fin)$ then, by the property of $\fB$, for any element $a\in\fA/\fin$
  there exists exactly one $B_a\in\fB$ such that $B_a^\bullet = a$.
  Therefore we can define $\rho(a)=B_a$ and $\rho$ is a homomorphism.
\end{proof}

\begin{thm}\label{4:2}
  If $\fin\sub \fA \sub P(\omega)$ is such an algebra that the quotient map $\fA\to\fA/\fin$ admits a lifting
  then the compactification $K_\fA$  of $\omega$ is smooth.
\end{thm}

\begin{proof}
  By Lemma  \ref{4:1} there exists an algebra $\fA$ of infinite
    subsets of $\omega$ such that $\fA = \alg(\fB\cup\fin)$.
 For every $n$  consider the ultrafilter $ p_n=\{B\in\fB\colon n\in B\}$ on $\fB$.
 Then $p_n$ extends to the nonprincipal ultrafilter $x_n$ on $\fA$, where
 \[ x_n=\{B\btu I\colon  B\in p_n, I\in\fin\}.\]

 It follows that $\delta_{x_n}-\delta_n\to 0$ on $\fA$ since for every $A\in\cA$
 we have $\delta_{x_n}(A)-\delta_n(A)=0$ except for finitely many $n$'s.
 Thus $K_\fA$ is smooth by Lemma \ref{3:1}.
\end{proof}

We note that thanks to Theorem \ref{4:2} one can easily define relatively big smooth compactifications of $\omega$.
Take for instance an independent sequence $\{B_\alpha\colon \alpha<\fc\}$ in $P(\omega)$ such that the set $ \bigcap_{\alpha\in I} B_\alpha^{\eps(\alpha)}$
is infinite for every finite $I\sub\fc$ and every $\eps\colon I\to\{0,1\}$.
Then the algebra $\fA$ generated by all $B_\alpha$'s and $\fin$ is such that $\fA/\fin$ has a lifting by Lemma \ref{4:1} and, by Theorem \ref{4:2}, $K_\fA$ is a smooth
compactification of $\omega$ which remainder is homeomorphic to the Cantor cube $2^\fc$. This can be generalized as follows.

\begin{thm}\label{4:3}
  If $L$ is a separable zerodimensional compact space
    then there exists a smooth compactification $\gamma\omega$ such that  $\growth$ is homeomorphic to $L$.
\end{thm}

\begin{proof}

We write $\clop(L)$ for the algebra of clopen subsets of $L$.
We define an embedding $\vf\colon   \clop(L)\to P(\omega)$ as follows.
Take a partition $\{B_d: d\in D\}$ of $\omega$ into infinite sets and let for $V\in\clop(L)$
\[\vf(V)= \bigcup_{d\in V\cap D} B_d.\]
Then for $U,V\in\clop(L)$, if  $\vf(U)=\vf(V)$ then $U\cap D=V\cap D$ so $U=V$. It is easy to check that
$\vf$ is indeed an isomorphic embedding of $\clop(L)$ into $P(\omega)$ and the algebra $\cB=\vf\left[\clop(L)\right]$ has the property
that every nonempty $B\in\fB$ is infinite.

Letting $\fA$ be the algebra in $P(\omega)$ generated by $\fB$ and $\fin$ we conclude from Theorem \ref{4:2} that
$K_\fA$ is a smooth compactification. Moreover, $K_\fA^*$ can be identified with $\ult(\fB)$ which is homeomorphic
to $\ult(\clop(L))$, so to $L$ itself.
\end{proof}

We prove below that there is a compactification $\gamma\omega$ which is not smooth but $\growth$ is separable
(and zerodimensional). The conclusion is that smoothness of a compactification $\gamma\omega$ cannot be decided by examining
topological properties of $\growth$ alone.

\begin{cor}\label{4:4}
There are two compactifications $\gamma \omega$ and $\eta\omega$ such that
$\growth\cong \eta\omega\sm\omega$, while
$\eta\gamma$ is smooth and  $\gamma\omega$ is not smooth.
\end{cor}

\begin{proof}
Take $\gamma\omega$ as in Theorem \ref{7:1}, that is a non-smooth zerodimensional compactification with $L=\growth$ separable.
 By Theorem \ref{4:3} there is a smooth compactification $\eta\omega$ such that $\eta\omega\sm\omega\cong L\cong \gamma\omega\sm \omega$.
\end{proof}

We finish the section by the following side remark. If $\fin\sub\fA\sub P(\omega)$ is an algebra
and the quotient map $\fA\to\fA/\fin $ admits a lifting then the algebra $\fA/\fin$ is $\sigma$-centred.
Note the reverse implication does not hold:
If we  take $\fA$ as in Corollary \ref{4:4} then $K_\fA$ is not smooth so $\fA$ does not have a lifting by Theorem \ref{4:2}.
 On the other hand, $\fA/\fin$ is $\sigma$-centred since it is isomorphic to the clopen algebra of a separable space $K_\fA^*$.

\section{The Measure algebra}\label{sec:5}

We start this section by the following observation due to W.\ Kubi\'s.

\begin{thm}\label{5:0}
If $\gamma\omega$ is a smooth compactification then its remainder carries a strictly positive measure.
\end{thm}

\begin{proof}
Take a sequence $(\nu_n)_n$ as in  Lemma \ref{2:3} and consider $\mu=\sum_n 2^{-n}|\nu_n|$. Then $\mu$ is a finite nonnegative measure
on $\gamma\omega$ and $\mu(\omega)=0$. Let $U\sub\gamma\omega$ be an open set such that $U\cap(\growth)\neq\emptyset$.
Take a continuous function $g\colon \gamma\omega\to [0,1]$ that vanishes outside $U$ and $g(x_0)=1$ for some $x_0\in U$.
Then the set $V=\{g>1/2\}$ contains infinitely many $n\in\omega$. Since $\nu_n(g)-g(n)\to 0$ we conclude that $\nu_n(g)>0$ for some $n$ and this gives
\[\mu(U)\ge\mu(g)\ge 2^{-n}|\nu_n|(g)>0,\]
so the measure $\mu$ is positive on every nonempty open subset of $\growth$.
\end{proof}

We  show in this section that under CH there are non-smooth compactifications $\gamma\omega$ such that $\growth$
carries a strictly positive nonatomic measure.

We consider here the measure algebra $\fM$, that is the quotient $Bor[0,1]/\cN$, where
$\cN$ is the ideal of Lebesgue null sets. We denote by $\lambda$ the measure on $\fM$ defined from the Lebesgue measure.
We write $S=\ult (\fM)$ for its Stone space.
It is well-known that the space $C(S)$ is isometric to $L_\infty[0,1]$.

By the classical Parovicenko theorem, under CH there is an isomorphic embedding   $\vf\colon \fM \to P(\omega)/\fin$.
Define an algebra $\fA\subseteq P(\omega)$ as the algebra of all
  finite modifications of elements of $\vf [\fM]$  and consider the
  Stone space $K_\fA$. Then $K_\fA$ is a compactification of $\omega$ such that $K_\fA^*$ is homeomorphic to $S$.

We shall  investigate the space $K_\fA$ in Theorem \ref{5:3}. Note first that $C(K_\fA^*)=C(S)$ is a Grothendieck space so it contains  no complemented copy of
$c_0$.

\begin{lemma}\label{5:1}
For every family $\{A_n\colon n\in\omega\}\sub\fA$ there is $B\in\fA$
which is an almost upper bound of that family, in the sense that
 $A_n\sub^* B$ for every $n$ and $B\cap A$ is
finite whenever $A\in\fA$ is such that $A\cap A_n=\emptyset$ for every $n$.
\end{lemma}

\begin{proof}
For every $n$ take $a_n\in\fM$ such that $\vf(a_n)=A_n^\bullet$. The algebra $\fM$ is complete so
the family $\{a_n\colon n\in\omega\}$ has the least upper bound $b\in\fM$. Now $B\in\fA$ such that
$B^\bullet=\vf(b)$ is as required.
\end{proof}

\begin{thm}\label{5:2}
Let $\fin\sub \fA\sub P(\omega)$ be an algebra such that $K_\fA^*\cong S$.  Then
$K_\fA$ is a compactification of $\omega$ that is not smooth.
\end{thm}

\begin{proof}
  By Lemma \ref{3:1} it is enough to  check that whenever $(\nu_n)_{n}$
is a   bounded sequence, where every $\nu_n\in\ba(\fA)$ vanishes on finite sets,
then $\nu_n-\delta_n$ does not converge to $0$.
Suppose otherwise; let $\nu_n(A) - \delta_n(A) \to 0$ for every $A\in\fA$.

Note first that there is an infinite $T\in\fA$ such that every $\nu_n$ is nonatomic on
the algebra $\fA_T=\{A\in\fA\colon  A\sub T\}$ of subsets of $T$. Indeed, every $\wh{\nu_n}$,
as a Radon measure on $K_\fA$ is concentrated on $S$ and the set $\{x\in S\colon  \wh{\nu_n}(\{x\})\neq 0\}$
is at most countable (since $|\wh{\nu_n}|(S)$ is finite; see Lemma \ref{3:4}). The space $S$ is not separable and therefore
there is a nonzero $a\in\fM$ such that $\wh{\nu_n}(\{x\})=0$ for every $n$ and every $x\in \wh{a}$.
Take $T\in\fA$ such that $T^\bullet =a$; then $T$ is as required.

Fix $\varepsilon =  1 / 8$.
  Take any  $A_1\in\fA$ such that $A_1\sub T$ and $A_1,T\sm A_1$ are infinite.
  Since for every $n\in A_1$ we have $\delta_n(A_1)=1$, so
    $\lim_{n\in A_1} \nu_n(A_1) = 1$.  Hence there exists $n_1 \in A_1$
    such that $\nu_{n_1}(A_1) \ge 1-\varepsilon$.
  Moreover, since the variation of $\nu_{n_1}$ is finite, there exists infinite $D_1\in\fA$ such that
  $D_1 \subseteq T\sm A_1$ and $|\nu_{n_1}|(D_1) < \eps$.

  In a similar way for every $k>1$ there exist infinite sets $A_k,D_k\in\fA$ and $n_1<n_2<\ldots$ such that

  \begin{enumerate}[(a)]
  \item $A_k \subseteq D_{k-1}$ and $D_k \subseteq D_{k-1}\sm A_k$;
  \item   $D_{k-1}\sm A_k$ is infinite;
  \item $n_k \in A_k$;
  \item  $\nu_{n_k}(A_k) \geq 1 - \eps$ and $|\nu_{n_k}|(D_k) < \eps$.
\end{enumerate}

  Since all  the measures $\nu_n$ are nonatomic on $T$, by Lemma \ref{3:4} for every $k\in \omega$
    there exists a set $B_k\in\fA$ such that $B_k \subseteq A_k$ and  $\big|\nu_{n_k}(B_k) - \frac 12\big| < \eps$.

As $A_k$ are pairwise disjoint and $n_k\in A_k$, it follows from $\nu_n-\delta_n\to 0$ that
there is an infinite $N\sub\omega$ such that for every $k\in N$
\begin{equation}\label{5:eq}
 \left| \nu_{n_k}\left(\bigcup_{j<k, j\in N} B_{j}\right)\right|<\eps.
 \end{equation}

Let $B\in\fA$ be an almost upper bound for $\{B_k: k\in N\}$ as in Lemma \ref{5:1}. Write, for simplicity,
$B_{<k}=\bigcup_{j<k, j\in N}B_j$, and let $B_{\le k}$ be defined accordingly.
For any $k\in N$ we have
  \[\nu_{n_k}(B)= \nu_{n_k}(B_{<k})+\nu_{n_k} (B_k)+  \nu_{n_k}(B\sm B_{\le k}),\]
  where
  $B\sm B_{\le k}  \sub^* D_k$.
  Using \ref{5:eq},  condition (d)  and the fact that  $\nu_{n_k}$ vanishes on finite sets we get
  \[ \big|\nu_{n_k}(B) - \frac 12\big| < 3\eps, \quad\mbox{so } 1/8< \nu_{n_k}(B)< 7/8,\]
for every $k\in N$ (note  that $1/2+3\eps=1/2+3/8=7/8$).

On the other hand,   consider the set $J = \{k\in N\colon n_k \in B\}$.
  If $J$ is infinite then $\nu_{n_k}(B) \to 1$ for $k \in J$.
  If $N\sm J$ is infinite we should have $\nu_{n_k}(B) \to 0$ for $k\in N\sm J$, and in both cases this is a contradiction.
\end{proof}

We shall now prove that for $\fA$ as in Theorem \ref{5:2} the space $C(K_\fA)$ need not to be Grothendieck.
  Consider the family $\fD$ of all subsets $A\subseteq\omega$ having the asymptotic density
  \[ d(A)= \lim_n |A\cap n| / n.\]
  Observe that the asymptotic density does not depend on finite
  modifications of the set $A\subseteq \fD$, so we can also treat $d$ as the
  function defined on $\fD/\fin$.

  In the proof of Theorem \ref{5:5} we shall use the following result of Frankiewicz and Gutek and a simple remark \ref{5:4}.

\begin{thm}[\cite{FG81}]\label{5:3}
  Assuming CH there exists an embedding $\vf\colon \fM \to P(\omega)/\fin$
    such that  $\vf(a) \in \fD/\fin$ and  $d(\vf(a)) = \lambda(a)$ for every $a\in\fM$.
\end{thm}

\begin{lemma}\label{5:4}
  There exists a family $\{I_n\}_{n\in\omega}$ of pairwise disjoint finite
  subsets of $\omega$ such that $d(A) = \lim_n\frac{|A\cap I_n|}{|I_n|}$
  for any $A\in\fD$.
\end{lemma}

\begin{proof}
Take any increasing sequence of integers $k_n$ such that $\lim_n k_n/k_{n+1}=0$ and put
$I_n=\{i\in\omega: k_n\le i< k_{n+1}\}$. Then $I_n$ are as required by standard calculations.
\end{proof}

\begin{thm}\label{5:5}
Under CH there is a compactification $\gamma\omega$ such that $\growth\cong S$ and
$C(\gamma\omega)$ contains a complemented copy of $c_0$.
\end{thm}

\begin{proof}
  We use an embedding $\vf\colon \fM\to P(\omega)/\fin$ as in  \ref{5:3} and consider $\fA$ as in the beginning of this section and
  in Theorem \ref{5:2}.
Take  a family of pairwise disjoint intervals $I_n$ as in Lemma \ref{5:4}.
  For every  $n\in\omega$ let $f_n = \chi_{I_n} \in C(K_\fA)$ be  the characteristic function of the set $I_n$.

  Let us consider the space
    $Y = \overline{\operatorname{lin}\la f_n\colon n\in\omega\ra}$.
  Then $Y$ is a closed subspace of $C(K_\fA)$ isomorphic to $c_0$; the isomorphism
 between them  is defined by setting $e_n \mapsto f_n$.

  Now $Y$ is complemented in $C(K_\fA)$; indeed,
  consider the measures
  \[\mu_n = \frac 1{|I_n|} \sum_{i\in I_n} \delta_i\in\ba_+(\fA).\]
  Then $\mu_n(f_k)$ equals 1 if $k=n$ and 0 if $n\neq k$ because $I_n$ are pairwise disjoint.
  Moreover, by Lemma \ref{5:4},  $\mu_n\to d$ on the algebra $\fA$ so we can apply Lemma \ref{2:1}.
\end{proof}

\section{Large smooth compactification}\label{sec:6}

\begin{thm}\label{6:1}
Under  CH  there exists a smooth compactification $\gamma\omega$ of $\omega$ such that $\growth$ is not separable.
\end{thm}

This section is devoted to proving  Theorem \ref{6:1}.
The desired compactification will be defined as $K_\fA$, where $\fA=\bigcup_{\alpha<\omega_1} \fA_\alpha$,
and countable algebras $\fA_\alpha$ are defined inductively. Lemma \ref{6:2} describes the staring point of the construction; we only sketch its proof here since
it closely follows  the proof of Sobczyk's theorem in Diestel \cite{Di84}.
Subsequent Lemma \ref{6:3} contains the essence of the inductive step

\begin{lemma}\label{6:2}
  There exist a countable nonatomic Boolean algebra $\fC\subseteq P(\omega)$
    and a~sequence of nonatomic
   probability  measures $(\nu_n)_{n}$ on $\fB = \alg(\fC\cup\fin)$
    such that $\nu_n|\fin\equiv0$ for every $n$ and $\nu_n-\delta_n\to0$ on $\fB$.
\end{lemma}

\begin{proof}
 There is an algebra $\fC\sub P(\omega)$ isomorphic to  $\clop(2^\omega)$, the algebra of clopen subsets of the Cantor set.
 We can copy the standard product measure  $\nu$ on $2^\omega$ onto $\fC$ (and denote it by the same letter).
  Put $\fB = \alg(\fC\cup\fin)$. Then we have a probability measure $\mu$ on $\fB$ defined by
  $\mu(C\btu I)=\nu(C)$ for every $C\in\fC$ and $I\in\fin$.

  Since $\fB$ is countable the space $K_\fB$ is metrizable, so $C(K_\fB)$ is a separable Banach space.
  Hence the unit ball
    $\cM_1(K_\fB)$ is metrizable in its $weak^*$ topology. Let the metric $\rho$ metrize  $\cM_1(K_\fB)$.
    Denote by $\cP$ the space of probability measures on $K_\fB$ that vanish on $\omega$.

    It is not difficult to check that the set of nonatomic $\nu\in \cP$ is $weak^*$ dense in $\cP$:
    consider for instance the convex hull of the family $\mu_C\in \ba_+(\fB)$, where $\mu_C(B)=(1/\mu(C))\cdot \mu(C\cap B)$ and $C\in\fC$.
    Thus for every $n$ we can choose nonatomic $\nu_n\in P$ such that
    \[ \rho(\nu_n,\delta_n)\le 2\cdot {\rm dist}(\delta_n,P).\]
  Then $\nu_n-\delta_n\to 0$ in the $weak^*$ topology.
\end{proof}

\begin{lemma}\label{6:3}
  Let $\fB\subseteq P(\omega)$ be a countable Boolean algebra containing $\fin$. Moreover, suppose that:
    \begin{enumerate}[(i)]
      \item we are given  a set $ \{ p_j\colon {j\in\omega}\}$ of ultrafilters which  dense in $K_\fB^*$;
      \item $(\nu_n)_{n\in\omega}$ is a sequence of nonatomic probability measures
        on $\fB$;
      \item $\nu_n(I) = 0$ for every $I\in\fin$ and every $n$;
      \item $\nu_n - \delta_n \to 0$ on $\fB$.
  \end{enumerate}

  Then there exists an infinite set $X\subseteq\omega$ such that
  \begin{enumerate}[---]
    \item for any extension $\widetilde p_j$ of $p_j$ to an ultrafilter on $\fB[X]$ the set
      $\{\widetilde p_j\colon j\in\omega\}$  is not dense in $K_{\fB[X]}^*$;
    \item we can extend every $\nu_n$  to  a probability measure $\widetilde\nu_n$ on $\fB[X]$ so that
      $\widetilde\nu_n - \delta_n \to 0$ on $\fB[X]$.
  \end{enumerate}
\end{lemma}

\begin{proof}
  Since $\fB$ is countable, we fix an enumeration $\fB = \{B_0, B_1, \ldots\}$.
   \medskip

 \noindent {\sc Claim.}  There are  infinite sets $A_j\in\fB$ and  $c(j)\in\omega$ for $j\in\omega$
    such that
  \begin{enumerate}[(1)]
    \item $\nu_0(A_j) <  2^{-(j+2)}$ for all $j\in\omega$;
    \item $A_j \in p_j$;
    \item for every $i$ and  $n\in\omega\sm A_j$ we have  $\nu_n(A_j)< 2^{-(j+2)}$;
    \item $c(j)\notin A_k$ for all $j,k$;
   \item for every $j$ either $B_j \subseteq \bigcup_{k\leq j} A_k$ or    $c(j)\in B_j$.
   \end{enumerate}

{\em Proof of Claim.} We proceed by induction on $j$.
Suppose that we have already constructed $A_0,\ldots, A_j$ and $c(0),\ldots c(j)$ and put $m=\max_{i\le j} c(i)+1$
(of course we can additionally assume that $0\notin A_i$).

  Since all the measures $\nu_i$ are nonatomic and $p_{j+1}$ is a nonprincipal ultrafilter, there is
  $A \in p_{j+1}$ such that $A\,\cap\, m=\emptyset$ and $\nu_i(A)<2^{-(j+3)}$ for every $i\le m$.
  Since $\nu_n(A)-\delta_n(A)\to 0$, the set $F=\{n\notin A: \nu_n(A)\ge 2^{-(j+3)}\}$ is finite.
  Define $A_{j+1}=A\cup F$. Since  $\nu_i(F)=0$ for every $i$,  (1) and (3) are granted by the choice of $A$.
  Condition (4) holds since $F\, \cap\, m=\emptyset$.

Now we can set $c(j+1)=0$ or choose $c(j+1)\in B_{j+1}\sm \bigcup_{k\le j+1}A_k$, if possible. This verifies Claim.
 \medskip

We take the sets $A_j$ from Claim and prove that  the set $X=\bigcup_j A_j$ is as desired.
We first check the following properties of $X$.

  \begin{enumerate}[(a)]
     \item For any $B\in\fB$ if $B\subseteq X$ then
     $B\subseteq \bigcup_{j\leq N} A_j$ for some $N\in\omega$.
    \item $\omega\sm X$ is infinite.
    \item $ \nu_n^*(X) = 1$ for every $n$.
    \item $\lim_{n\notin X} (\nu_n)_* (X) = 0$.
  \end{enumerate}

Ad  (a). Take any $B_j\in \fB$ such that $B_j\subseteq X$. Suppose that
$B_j\not\subseteq \bigcup_{i\leq j}A_i$.
Then $c(j)\in B_j\sm X$ by (4) and (5), a contradiction.

  Ad (b). Suppose $\omega\sm X$ is finite.
  Then, since every finite set belongs to the algebra $\fB$,
    $X\in\fB$.
  By (a) $X\subseteq \bigcup_{i\leq N} A_i$ for some $N$.
  Since $\nu_0$ is insensitive to finite modifications of sets,
    by condition (1)
    \[\nu_0(\omega) = \nu_0(X)\le
      \nu_0\left(\bigcup_{i\leq N} A_i\right) \leq
      \sum_{i\leq N} 2^{-(i+2)} \leq 1/2,\]
    which contradicts the fact that $\nu_0$ is a probability measure
    on $\omega$.

 Ad  (c). For any $n$, if $B\in\fB$ and $\nu_n(B)>0$ then $B$ is infinite so
 $B\in p_j$ for some $j$ and hence $B\cap X\neq\emptyset$. This proves $\nu_n^*(X)=1$.

 Ad (d). By (a) every $B\in\fB$ that is contained in $X$ is in fact contained in some finite union of $A_j$'s; hence
 \[ (\nu_n)_*(X)=\sup_N \nu_n\left( \bigcup_{i\le N} A_i\right).\]
 Observe that by condition (3), if $n\notin X$ and $k<N$ then
    \begin{equation}\label{eq:inner_measure}
      \nu_n\left(\bigcup_{i=k}^N A_i\right) \le
      \sum_{i=k}^N \nu_n(A_i) \le
      \sum_{i=k}^N 2^{-(i+2)} < 2^{-(k+1)}.
    \end{equation}

Using the above estimate we can compute the limit of inner measures:
    \begin{align*}
      \lim_{n\notin X} {\nu_n}_*(X) &=
      \lim_{n\notin X} \sup_N \nu_n\big(\bigcup_{i\leq N} A_i\big) \\ &\leq
      \lim_{n\notin X} \sup_N \Bigg(\nu_n\bigg(\bigcup_{i<k} A_i\bigg) +
        \nu_n\left(\bigcup_{i=k}^N A_i\right)\Bigg) \\ &\leq
      \lim_{n\notin X} \Bigg(\nu_n\bigg(\bigcup_{i<k} A_i\bigg) +
        2^{-(k+1)}\Bigg) \text{\scriptsize(by \eqref{eq:inner_measure})}\\ &=
      \lim_{n\notin X}\nu_n\bigg(\bigcup_{i<k} A_i\bigg) + 2^{-(k+1)}.
    \end{align*}
  But for any $k\in\omega$ the set $\bigcup_{i<k}A_i\in\fB$ so we
    have $\lim_{n\notin X}\nu_n(\bigcup_{i<k}A_i) = 0$.
  Since $k$ is arbitrary, this proves (d).
\medskip

Once we know that  $X$ satisfies (a)-(d),  we can check that $X$ is indeed the set
  are looking for.
Let, for every $j$, $\widetilde p_j \in\ult(\fB[X])$ be an
    arbitrary extension of $p_j$.
  Because $X\supseteq A_j$ and $A_j \in p_j$ so $X \in \widetilde p_j$.
  Thus $\omega\sm X$ omits all  the ultrafilters $\widetilde p_j$ and, since it is an
    infinite set,
    \[\wh{\omega\sm X}\cap K_{\fB[X]}^*\neq\emptyset,\]
     which indicates that $\widetilde p_j$ are not dense in $K_{\fB[X]}^*$.

  Now appealing to Theorem \ref{3:2} we define the measures $\widetilde\nu_n$ on $\fB[X]$ extending $\nu_n$ so that
  \begin{equation}
  \widetilde{\nu_n}(X) = \begin{dcases*}
  \nu_n^*(X) & for $n\in X$, \\
  (\nu_n)_*(X) & for $n\notin X$.
  \end{dcases*}
  \end{equation}

  Then for $n\in X$ we have $\widetilde\nu_n(X) = \nu_n^*(X) \to 1$ and for
  $n\notin X$ we have $\widetilde\nu_n(X) = (\nu_n)_*(X) \to 0$
  by property (d).
  Using Lemma \ref{3:3} we conclude that
    $\widetilde{\nu_n} - \delta_n \to 0$ on $\fB[X]$, and the proof is complete.
\end{proof}

We have already all essential ingredients to carry out a diagonal construction leading to Theorem \ref{6:1}.

\begin{proof}[Proof of Theorem \ref{6:1}]
Let $\fA_0$ be the algebra from Lemma \ref{6:2} and let $(\nu_n^0)_n$ be  a sequence
of nonatomic probability measures on $\fA_0$ such that $\nu^0_n-\delta_n\to 0$.

We construct inductively, for $\xi<\omega_1$,  a sequence of countable algebras $\fA_\xi\sub P(\omega)$, sets $X_\xi\sub\omega$ and
sequences $(\nu_n^\xi)_n$ of probability measures on $\fA_\xi$ such that

\begin{enumerate}[(i)]
\item $\fA_\beta\sub \fA_\xi$ for all $\beta<\xi<\omega_1$;
\item $\fA_\xi$ is generated by $\bigcup_{\beta<\xi} \fA_\beta$ and $X_\xi$;
\item $\nu^\xi_n\,|\, \fA_\beta=\nu^\beta_n$ for every $n$ and $\beta<\xi$;
\item $\nu^\xi_n-\delta_n\to 0$ on $\fA_\xi$ for every $\xi$.
\end{enumerate}

Then we consider the algebra $\fA=\bigcup_{\xi<\omega_1}\fA_\xi$;  for every $n$ let $\mu_n$ be the unique
probability measure on $\fA$ such that $\mu_n|\fA_\xi=\nu^\xi_n$ for $\beta<\xi$. It is clear that $\mu_n-\delta_n\to 0$ on $\fA$
so $K_\fA$ is a smooth compactification of $\omega$ by Lemma \ref{3:1}.
Therefore it is enough to check that by a suitable choice of sets $X_\xi$, we can guarantee  that $K_\fA^*$  is not  separable.

Using CH we fix an enumeration $\{D(\xi): \alpha<\omega_1\}$ of all countable
dense sets in $K_{\fA_0}^*$.
At step $\xi$ we let $\fB=\bigcup_{\alpha<\xi} \fA_\alpha$ and consider a sequence of measures $\nu_n$ defined on $\fB$, where
 $\nu_n$ is the unique extensions of $\nu^\alpha_n$, $\alpha<\xi$.
 Then we apply Lemma \ref{6:3}
to find a set $X_\xi$ such that (iv) is granted for $\fA_\xi=\fB[X_\xi]$ and, at the same time $X_\xi$ witnesses
any extensions of ultrafilters from $D(\xi)$ are no longer dense in $K_{\fA_\xi}^*$.

If follows $K_\fA^*$ is not separable. Indeed, if we had a countable dense set $D\sub K_\fA^*$ then
the set $D=\{x|\fA_0:x\in X\}$  would be dense  in $K_{\fA_0}^*$. But $D=D_\xi$ for some $\xi<\omega_1$ and
$D_\xi$ is not dense in $\fA_\xi$.
\end{proof}

\section{Small and ugly}\label{sec:7}

We construct in this section a relatively small compactification $\gamma\omega$
which is not smooth, contrastive with the compactification from section \ref{sec:6}.

\begin{thm}\label{7:1}
  There exists a non-smooth compactification $\gamma\omega$ which is first-countable and  which
   remainder $\growth$ is separable.
\end{thm}

We again construct a certain algebra $\fA\sub P(\omega)$; this time the main idea is to keep a fixed countable
dense set of ultrafilter and to kill all possible sequences of measures that would witness the smoothness.

We shall use  the notion of minimal extensions of  algebras introduced by Koppelberg \cite{Ko89} which we recall now
in the context of  subalgebras od $P(\omega)$. The basic facts we list below can be found in \cite{Ko89} or \cite{BN07}.

If $\fA\sub P(\omega)$ and $X\sub\omega$ then $\fA[X]$ is said to be
a {\em minimal extension} of $\fA$ if for any algebra $\fB$, if $\fA\sub\fB\sub \fA[X]$ then either $\fB=\fA$ or $\fB=\fA[X]$.
This is equivalent to saying that for every $A\in\fA$, either $ X\cap A \in\fA$ or $X\sm A\in\fA$.

If $\fA[X]\not=\fA$   is a minimal extension then there is exactly  one $q(X)\in\ult(\fA)$ that gets split in $\fA[X]$; this is
$q(X)=\{A\in\fA\colon A\cap X\notin\fA\}$.
Then  every ultrafilter $p\neq q(X)$ has a unique extension to $\widetilde p\in\ult(\fA[X])$.

Given a sequence $A_n\in\fA$ and $p\in\ult(\fA)$,
it will be convenient to say that  $A_n$ converge to $p$ if every $B\in p$ contains almost all $A_n$.

\begin{lemma}\label{7:1.1}
If $\fA[X]$ is a minimal extension of $\fA$ and $\mu\in\ba_+(\fA)$ does not have an atom
at $q(X)$ then $\mu$ has a unique extension  to $\widetilde \mu\in\ba_+(\fA[X])$.
\end{lemma}

\begin{proof}
For every $\eps>0$ there is $A\in q(X)$ with $\mu(A)<\eps$. Then $B=X\sm A\in\fA$ so we have $B\sub X\sub B_1=B\cup A$,
where $B,B_1\in\fA$ and $\mu(B_1\sm B)<\eps$. This shows that  $\mu_*(X)=\mu^*(X)$ must be equal to $\widetilde\mu(X)$, whenever
$\widetilde\mu$ is a nonnegative extension of $\mu$.
We can repeat this argument with $A\cap X$ and $A\sm X$ for $A\in\fA$ to conclude that $\widetilde\mu=\mu^*=\mu_*$ on $\fA[X]$.
\end{proof}

\begin{lemma}\label{7:2}
Let $\fB \subseteq P(\omega)$ be a  algebra containing $\fin$ with $\fB/\fin$ nonatomic.
  Let $C=\{p_j\colon  j\in \omega\}$  be a dense subset of $K_\fB^*$. Further let

  \begin{enumerate}[(i)]
  \item $q\in K^*_\fB\sm C$ be a point of countable character;
  \item $(n_k)_k$ be a sequence on $\omega$ such that $n_k\to q$ (in the space $K_\fB$);
  \item $(B_k)_k$ and $(D_k)_k$ be sequences in $\fB$ of infinite sets that converge to $q$ and such that $B_k\cap D_j=\emptyset$ for all $j,k$.
  \end{enumerate}
  If we let
  \[ X=\{n_k\colon  k\in\omega\}\cup \bigcup_k B_k,\]
  then $\fB[X]$ is a minimal extension of $\fB$ and only $q=q(X)$ may be split in $\fB[X]$.
  Consequently, every $p_j$ has a unique extension to an ultrafilter $\widetilde p_j$ on
  $\fB[X]$ and $\widetilde C=\{\widetilde p_j: j\in\omega\}$ is dense in $K_{\fB[X]}^*$.
 \end{lemma}

\begin{proof}
For every $B\in q$, we have $X\sm B\in\fB$ since $B_k\cup\{n_k\}$ converge to $q$.
Hence $\fB[X]$ is a minimal extension.

To complete the proof we have to check the density of $\widetilde C$, i.e.\
that for every $B\in\fB$, if $B\cap X$ is infinite then $B\cap X\in\widetilde p_j$ for some $j$,
and if $B\sm X$ is infinite then $B\sm X \in\widetilde p_j$ for some $j$ (for any $B\in\fB$).

Take $B\in\fB$ such that $B\cap X$ is infinite. If $B\in q$ then $B_k\sub B\cap X$ for some $k$ and, taking $j$ with
$B_k\in p_j$, we get $B\cap X\in \widetilde p_j$. If $B\notin q$ then $B\cap X\in\fB$ (by (ii) and (iii)) so $B\cap X\in p_j$ for some $j$.

Suppose now that $B\sm X$ is infinite. If $B\notin q$ then $B\sm X\in\fB$, as above.
Finally, if $B\in q$ then there is $k$ such that $D_k\sub B$. Then $D_k\sub B\sm X$ by (iii) and, taking $j$ with
$D_k\in p_j$, we get   $B\sm  X\in \widetilde p_j$, so the proof is complete.
\end{proof}

Here comes the lemma which constitutes the essence of the inductive step.

\begin{lemma}\label{7:3}
  Let $\fB \subseteq P(\omega)$ be an algebra such that $\fB/\fin$ is nonatomic and $K_\fB^*$ is first-countable.
  Let $( p_j)_ j$ be a sequence of nonprincipal
    ultrafilters on $\fB$ such that the set $C=\{p_j\colon j<\omega\}$ is dense
    in $K_\fB^*$.
  Suppose also that $(\nu_n)_n$ is a sequence of measures on
    $\fB$ such that $\nu_n|\fin\equiv 0$ for every $n$ and $\nu_n - \delta_n \to 0$ on $\fB$.

  Then there exists a set $X\subseteq \omega$ such that
  \begin{enumerate}[(1)]
    \item every $p_j$ has a unique extension to $\widetilde p_j\in\ult(\fB[X])$ and
   the set  $\{\widetilde p_j \colon j<\omega\}$ is dense in $K_{\fB[X]}^*$;
    \item if $\bar\nu_n$ is an extension of $\nu_n$ to a measure on $\fB[X]$ and
    $||\bar\nu_n||=||\nu_n||$ for every $n$ then
      $\bar\nu_n(X) - \delta_n(X) \not\to 0$.
   %   then $\sup\{\|\bar\nu_n\|\colon n\in\omega\} \geq M + 1$.
  \end{enumerate}
\end{lemma}

\begin{proof}
  We shall choose an ultrafilter $q$ on $\fB$  and construct  $B_n,D_n\in\fB$ and
  numbers $n_k$ as in Lemma \ref{7:1} so that the set
    $X = \{n_1, n_2, \ldots\} \cup \bigcup_n B_n$ satisfies condition (2).
  Then (1) is granted by Lemma \ref{7:2}.

  Take a point $q$ in $K_\fB^*\sm C$ which is  an atom of no measure $\nu_n$.
  Let $\{U_k:k\in\omega\}\sub\fB$ be a base at $q\in K_\fB$.
  % such that
  % \begin{equation} \label{eq2}
  %(\forall j\le k) |\nu_j|(U_k)<1/k.
  %\end{equation}
  Choose also a sequence $(m_i)_i$ in $\omega$ such that $m_i\to q$ in the space $K_\fB$ and
  put $N=\{m_i: i\in\omega\}$.

  We construct inductively $C_k,D_k,V_k \in\fB$ and  $n_k \in N$ such that for every $k$

   \begin{enumerate}[(a)]
    \item $V_k\sm (B_k\cup D_k)\in q$ and $B_k\cup D_k\subseteq V_k\sub U_k$;
    \item $n_k\in V_{k-1} \supseteq V_{k}$;
    \item $B_i \cap D_j = \emptyset$ for all $i,j$;
   \item $\nu_{n_k} \left( \bigcup_{j<k} B_{j}\right)< 1/k$;
   \item $|\nu_{n_k}|(V_k)<1/k$.
  %  \item $|\nu_{n_k}|(U_k)<1/k$.
  \end{enumerate}

  The inductive construction is straightforward: if we carried it out  below $k$ then set
  \[ B=\bigcup_{j\le k-1} B_j,\quad D=\bigcup_{j\le k-1} D_j, \quad V=V_{k-1}\cap U_k \sm (B\cup D).\]
   Then   $V\in q$ so $V$ contains infinitely many $m_i$ (as
   $m_i\to q$) and among them we choose
   $n_k$ so that (d) holds (using $\nu_n-\delta_n\to 0$).
   Then we choose $V_k\sub V$ satisfying (e) (as $q$ is not an atom of  $\nu_{n_k}$).
  Finally, we can choose  $B_k,D_k\sub V_k\sm (B\cup D)$ so that (a),  (b)  and (c) hold.
\medskip

   Recall that $X = \{n_1, n_2, \ldots\} \cup \bigcup_n  B_n$.
   Consider now any extensions of measures $ \nu_{n}$ to $ \bar\nu_{n}\in\ba(\fB[X])$
   with $||\bar \nu_n||=||\nu_n||$. Note that in such a case $|\bar \nu_n|$ is an extension of $|\nu_n|$.
 Since we have for any $k$ (using (a))
   \[
    \bigcup_{n\in\omega} B_n = \bigcup_{j< k}B_j \cup \bigcup_{j\ge k} B_j
      \subseteq \bigcup_{j < k}B_j\cup V_k,\]
   and the set $N\sm V_k$ is finite,  we get
  \[  \bar\nu_{n_k}(X) \le
    \nu_{n_k}\left(\bigcup_{j < k} B_j\right) + |\bar\nu_{n_k}|(V_k)< 1/k+ |\bar\nu_{n_k}|(V_k)=1/k+ |\nu_{n_k}|(V_k)<2/k.
\]
But $n_k\in  X$, so $\nu_n(X)-\delta_n(X)\not\to 0$, and we are done.
\end{proof}

\begin{proof}[Proof of Theorem \ref{7:1}.]
 We first describe a certain operation  that will be iterated to construct an algebra defining the required compactification.

 Consider and algebra $\fB$  in $P(\omega)$ containing $\fin$ and such that $\fB/\fin$ is nonatomic.
Suppose also that

\begin{enumerate}[(1)]
\item $K_\fB$ is first-countable;
\item   $C=\{p_j: j\in \omega\}$ is a dense subset of  $K_{\fA_0}^*$;
 \item $\ba(\fB)$ is of size $\fc$ and $\nu^{\xi}= (\nu^{\xi}_n)_n$, $\xi<\fc$ is an enumeration  of all bounded
 sequences of measures on $\fB_\xi$ such that $\nu^{\xi}_n -\delta_n\to 0$ on $\fB_\xi$.
 \end{enumerate}

 We fix a set $Q\sub K_\fB^*\sm C$ of cardinality $\fc$ and for every $\xi<\fc$
 apply Lemma \ref{7:3} to the sequence $\nu^\xi$: pick $q_\xi\in Q\sm\{q_\eta:\eta<\xi\}$ and form the
 set $X_\xi$ as in Lemma \ref{7:3}.

 Then we let $\fB^\#$ be the algebra generated by $\fB$ and $\{X_\xi:\xi<\fc\}$. Note that

 \begin{enumerate}[(a)]
 \item $K_{\fB^\#}$ is first-countable;
 \item every $p_j$ has a unique extension to $ p_j^\#\in\ult(\fB^\#)$;
 \item  $C^\#=\{p_j^\#\colon j\in\omega\}$ is dense in $K_{\fB^\#}^*$;
 \item $\left|\ba(\fB^\#)\right|=\fc$.
 \end{enumerate}

 Ad (a). If $p\in\ult(\fB)$ is never split then it has a base in $\fB$. Otherwise $p=q_\xi$ for some $\xi<\fc$ and it is split only by $X_\xi$ into
 two ultrafilters having bases in $\fB[X_\xi]$.

 Ad (b) and (c). This follows from Lemma \ref{7:2}

 Ad (d). It suffices to note that any $\mu\in \ba_+(\fB)$ has at most $\fc$ extensions to nonnegative measures on $\fB^\#$.
Indeed,  take $\mu\in\ba_+(\fB)$ and let $N\sub K_\fb^*$ be the set of all atoms of $\mu$. Then $N$ is countable and
 the algebra $\fB'$ generated by $\fB$ and $\{X_\xi: q_\xi\in N\}$ is countably generated over $\fB$.
 Therefore we can extend $\mu$ to a nonnegative measure $\mu'$ on $\fB'$ in at most $\fc$ ways. In turn every such
 $\mu'$ extends uniquely to a measure in $\ba_+(\fB^\#)$ by Lemma \ref{7:1.1}.

 We shall now iterate the operation $\#$. Let  $\fA_0$ be a countable algebra in $P(\omega)$ such that $\fA_0/\fin$ is nonatomic.
 Fix $C=\{p_j\colon j\in\omega\}$ as above and choose a pairwise disjoint family $\{Q_\alpha\colon \alpha<\omega_1\}$
 such that $|Q_\alpha|=\fc$ and $Q_\alpha\sub K_{\fA_0}^*\sm C$ for every $\alpha<\omega_1$.

 We define $\fA_{\alpha+1}=(\fA_\alpha)^\#$, with $Q_\alpha$ playing the role of $Q$ in the construction.
 We also let $\fA_\alpha=\bigcup_{\beta<\alpha}\fA_\beta$ for $\alpha<\omega_1$ limit, and
 claim that $\fA=\bigcup_{\alpha<\omega_1}\fA_\alpha$ is the required algebra.

The compactification   $K_\fA$ is not  smooth by Lemma \ref{3:1}.
Indeed, take any $\mu_n\in \ba(\fA)$ with $\sup_n||\mu_n||<\infty$.
 Note that every measure on $\fA$ attains its norm on some countable subalgebra. Hence
 there is $\alpha<\omega_1$ such that, writing $\nu_n=\mu_n|\fA_\alpha$, we have
  $||\nu_n||=||\mu_n||$ for every $n$.
 Then $(\nu_n)_n=\nu^{\xi}$ for some $\xi<\fc$ and by our construction
 $\nu_n(X_\xi)-\delta_n(X_\xi)\not\to 0$.

Finally, $K_\fA^*$ is separable because $C$ remains dense throughout the construction. The fact that $K_\fA$ is first-countable
follows from the fact that $Q_\alpha$'s are pairwise disjoint, as in the proof of (a) above.
\end{proof}

In the terminology of \cite{Ko89},
if $\fB\sub\fA\sub P(\omega)$ then $\fA$ is minimally generated over $\fB$ if, for some $\xi$, $\fA$ is a continuous increasing union
$\fA=\bigcup_{\alpha<\xi} \fA_\alpha$, where $\fA_0=\fB$ and $\fA_{\alpha+1}$ is a minimal extension of $\fA_\alpha$ for every $\alpha<\xi$.
An algebra  $\fA$ is minimally generated if it is minimally generated over the trivial algebra.
It is clear  that the algebra $\fA$ we  constructed in the proof of Theorem \ref{7:1} is minimally generated.

In the language of extension operators, Theorem \ref{7:1} says the following.

\begin{cor}\label{7:5}
There exist a separable first-countable compact space $L$ and a compact superspace $K$ with $K\sm L$ countable such that
there is no extension operator $C(L)\to C(K)$.
\end{cor}

\begin{proof}
Take $\gamma\omega$ from Theorem \ref{7:1}, put $L=\growth$, $K=\gamma\omega$ and apply
Lemma \ref{2:3.1}.
\end{proof}

\section{On hereditarily Sobczyk spaces}\label{sec:8}

In this final section we prove a general result on compacta $K$ for which the space $C(K)$ contains a rich family of
complemented copies of $c_0$. In the definition below we use the terminology from D\v{z}amonja and Kunen \cite{DK95}.

\begin{mydef}\label{8:1}
We say that a compact space $K$ is in the class (MS) (of measure separable spaces) if every probability Radon measure $\mu$ on $K$
has the countable Maharam type, i.e.\ $L_1(\mu)$ is a separable Banach space.
%  A measure $\mu\in P(K)$ is of \emph{Maharam type} $\omega$ if
 %   there exists countable family $\cC \subseteq \operatorname{Bor}(K)$
  %  such that for any Borel set $B \subseteq K$ and any $\varepsilon > 0$ there
   % is $C\in\cC$ such that $\mu(B\bigtriangleup C) < \varepsilon$.
 % Equivalently, is separable.
\end{mydef}

Recall also the following standard notion.

\begin{mydef}\label{8:2}
  Let $X$ be any vector space and let $(x_n)_n $ be a sequence in $X$.
  A sequence $(y_n)_n$ is a \emph{convex block   subsequence} of  $(x_n)_n$  if there exist finite sets
  $I_n \subset \omega$ with $\max I_n<\min I_{n+1}$, and a function $a:\omega\to\er_+$ such that for all  $n\in \omega$
 \[y_n = \sum_{j\in I_n} a(j)\, x_j \quad \mbox{and} \quad \sum_{j\in I_n} a(j) = 1.\]
\end{mydef}

Our result is a consequence of (a particular case of) a result due to Haydon, Levy and Odell \cite{HLO87}, see also Krupski and Plebanek \cite{KP11} for a direct approach
to the following.

\begin{thm}[Haydon, Levy, Odell]\label{8:3}
  If $K$ is a~compact space in the class (MS) then every bounded sequence $(\mu_n)_n$
in $\cM(K)$ has a~convex block subsequence   $(\nu_n)_n$  converging  to some measure $\nu\in \cM(K)$.
\end{thm}

\begin{thm}\label{8:4}
  Let $K$ be a compact space in (MS).
  Then for any isomorphic embedding $T\colon c_0 \to C(K)$ the space $T[c_0]$ contains a subspace $Y$ which is isomorphic to $c_0$
  and complemented in $C(K)$,
 \end{thm}

\begin{proof}
Let $(e_n)$ be the sequence of unit vectors in  $c_0$; we write $e_n^*\in c_0^* = l_1$.

  Given an isomorphic embedding $T\colon c_0 \to C(K)$, put  $g_n = T e_n$ for every $n$.
  Since $T$ is an embedding, the dual operator
  \[T^*\colon C(K)^* =\cM(K) \to c_0^* = l_1,\]
    is onto and therefore there is a bounded sequence of   measures $\mu_n \in\cM(K)$ such that $T^* \mu_n = e_n^*.$
Then we have $\mu_n(g_k) = \mu_n(T e_k)= T^* \mu_n (e_k) = e_n^* (e_k)$.

  Consider the sequence of measures $(\mu_n)_n$. By Theorem \ref{8:3} it has a convex block subsequence  $(\nu_n)_n$
    converging to some measure $\nu \in M(K)$. Say that $\nu_n = \sum_{k\in I_n} a(k) \mu_k$,
where $I_n$ and $a$ are as in Definition \ref{8:2}.

For every $n$ put $\bar e_n=\sum_{k\in I_n} e_k$. Then $\bar e_n$ are norm-one vectors spanning a subspace $X$ of $c_0$ that
is clearly isometric to $c_0$.
  Hence the functions $h_n = T\bar e_n\in C(K)$ span a subspace $Y=T[X]$ of $T[c_0]$ that is isomorphic to $c_0$ and it is enough to
  check that $Y$ is complemented in $C(K)$.

  Since $T^*\mu_n=e_n^*$, we have $T^*\nu_n=\sum_{i\in I_n}t(i)e_i^*$, and
  \[\nu_n(h_k)=\nu_n(T\bar e_k)=T^*\nu_n (\bar e_k)=\sum_{i\in I_n} t(i)e_i^*\left( \sum_{j\in I_k} e_j\right)=\sum_{i\in I_n, j\in I_k} t(i)e_i^*(e_j),\]
  which is equal to 0 if $n\neq k$ (since then $I_n\cap I_k=\emptyset$), and is equal to $\sum_{i\in I_n} t(i)=1$ when $n=k$.

  Now, as in Lemma \ref{2:1}, we conclude that $P:C(K)\to C(K)$ defined by
  \[ Pf=\sum_{n\in\omega} (\nu_n-\nu)(f)\cdot h_n,\]
 is a bounded projection onto $Y$. Indeed,
 \[ \nu(h_n)=\lim_j \nu_j(h_n)=\lim_j T^*\nu_j (\bar e_n)=\lim_j \left(\sum_{i\in I_j}t(i)e_i^*\right) (\bar e_n)=0,\]
 for every $n$. This shows that $Ph_n=h_n$; moreover, $Pf\in Y$ for any $f\in C(K)$ since
  since $\nu_n(f)-\nu(f)\to 0$  for every $n$.
\end{proof}

A Banach space $X$ having the property that every isomorphic copy of $c_0$ in $X$ has a subspace isomorphic to $c_0$ and complemented in $X$ is
called {\em hereditarily separably Sobczyk} in \cite{GP03}. Theorem \ref{8:4} states that $C(K)$ is such a space whenever $K$ is in the class (MS).
Molt\'o \cite{Mo91} gave an example of a Rosenthal compact space $K$ such that $C(K)$ does not have the Sobczyk property. Rosenthal
compacta are in (MS), due to a results of Bourgain and \stevo, see \cite{DP15} for a more general result and references therein.
Consequently, $C(K)$ is hereditarily separably Sobczyk whenever $K$ is Rosenthal compact.

The final result is related to our Theorem \ref{7:1}.

\begin{cor}
If $\fA$ is a minimally generated Boolean algebra then $C(K_\fA)$ is hereditarily separably Sobczyk.
\end{cor}

\begin{proof}
This follows from Theorem \ref{8:4} and the remark after it,  and a result due to Borodulin-Nadzieja \cite{BN07} stating that
$K_\fA$ is in the class (MS) whenever $\fA$ is minimally generated.
\end{proof}

\end{document}